\documentclass[12pt, twoside]{amsart}
\usepackage{mathrsfs,amsthm,amscd,amssymb}
\usepackage{color}
\newcommand{\xr}[1]{\textcolor{red}{#1}}

\newtheorem{thm}{Theorem}[section]
\newtheorem{lem}[thm]{Lemma}
\newtheorem{cor}[thm]{Corollary}
\newtheorem{prop}[thm]{Proposition}

\theoremstyle{definition}
\newtheorem{defn}[thm]{Definition}
\newtheorem{exmp}[thm]{Example}
\newtheorem{rem}[thm]{Remark}
\newtheorem*{ack}{Acknowledgments}
\newcommand{\Nqlc}[0]{{\operatorname{Nqlc}}}
\newcommand{\Supp}[0]{{\operatorname{Supp}}}

\title[Effective freeness]
{The Angehrn--Siu type effective
freeness for quasi-log 
canonical pairs}
\author{Haidong Liu}
\date{\xr{2015/1/5, 16:00, version 0.20}}
\keywords{effective freeness, effective 
point separation, quasi-log canonical pairs, 
semi-log canonical pairs, inversion of 
adjunction}
\address{Department of Mathematics, Graduate School of Science,
Kyoto University, Kyoto 606-8502, Japan}
\email{liu.dong.82u@st.kyoto-u.ac.jp}

\begin{document}
\maketitle

\begin{abstract}
We prove the Angehrn--Siu type effective freeness and 
effective point separation for quasi-log canonical 
pairs. As a natural consequence, we obtain that these two 
results hold for semi-log canonical pairs. 
One of the main ingredients of our proof is 
the inversion of adjunction for quasi-log canonical 
pairs, which is established in this paper.
\end{abstract}


\section{Introduction}\label{sec1}
The theory of mixed Hodge structures plays an important 
role in the recent developments of the minimal model 
program (see, for example, \cite{ks}). 
Now we have various powerful vanishing theorems 
based on the theory of mixed Hodge structures on cohomology 
with compact support (see \cite{fuji0}, \cite{fuji4}, and so on). 
They are much sharper than the Kawamata--Viehweg 
vanishing theorem and the (algebraic version of) Nadel 
vanishing theorem. By these new vanishing theorems, 
the fundamental theorems of the minimal 
model program were established for quasi-log canonical 
(qlc, for short) 
pairs (see \cite{fuji0}, \cite{fuji4}, and so on). 
Note that the notion of quasi-log structures was first introduced 
by Ambro in \cite{ambro}. The category of qlc 
pairs is very large and contains kawamata log 
terminal pairs, log canonical pairs, quasi-projective 
semi-log canonical pairs (see \cite[Theorem 1.1]{fuji6}), and 
so on. The notion of qlc pairs seems to be indispensable 
for the cohomological study of semi-log canonical 
pairs (see \cite{fuji6}). 
In this paper, we formulate the Angehrn--Siu type 
effective freeness and effective point separation 
for qlc pairs and prove them in the framework of 
quasi-log structures. Of course, our results 
generalize \cite{as}, \cite[5.8, 5.9]{kollar}, 
and \cite[Theorems 1.1 and 1.2]{fuji1}.
The effective freeness for qlc pairs is as follows.

\begin{thm}[Effective freeness]\label{thm1.1}
Let $[X,\omega]$ be a 
projective qlc pair such that 
$\omega$ is an $\mathbb R$-Cartier 
divisor and let $M$ be a Cartier divisor 
on $X$ such that $N=M-\omega$ is ample. 
Let $x\in X$ be a closed point. We assume 
that there are positive numbers $c(k)$ with the following properties.
\begin{itemize}
\item[$(1)$] If $x\in Z\subset X$ is an 
irreducible {\em{(}}positive dimensional{\em{)}} subvariety, then
$$
N^{\dim Z} \cdot Z>c(\dim Z)^{\dim Z}.
$$
\item[$(2)$] The numbers $c(k)$ satisfy the inequality:
$$
\sum \limits_{k=1}\limits^{\dim X} \frac{k}{c(k)}\leq 1.
$$
\end{itemize}
Then $\mathscr O_X(M)$ has a global section not vanishing at $x$.
\end{thm}

A key ingredient of the proof of Theorem \ref{thm1.1} is 
the inversion of adjunction for qlc pairs (see Theorem \ref{thm2.10}). 
We will formulate and prove it in Section \ref{sec2}.  

\begin{rem}\label{rem1.2} 
In Theorem \ref{thm1.1}, 
we have 
$H^{1}(X, \mathscr{I}_{W}\otimes \mathscr{O}_{X}(M))=0$, 
where $W$ is the minimal qlc stratum of $[X,\omega]$ 
passing through $x$ and 
$\mathscr{I}_{W}$ is the defining ideal 
sheaf of $W$ on $X$ (see Theorem \ref{thm2.7}). Therefore, 
the natural restriction map 
$$H^{0}(X, \mathscr{O}_{X}(M))\rightarrow 
H^{0}(W, \mathscr{O}_{W}(M)) $$ 
is surjective. Thus, by replacing $X$ with $W$, 
we can assume that $X$ is irreducible in Theorem \ref{thm1.1}. 
\end{rem}

By suitably modifying the proof of Theorem \ref{thm1.1}, 
we can prove the following effective point separation 
for qlc pairs without any difficulties. 

\begin{thm}[Effective point separation]\label{thm1.3}
Let $[X,\omega]$ be a 
projective qlc pair such that 
$\omega$ is an $\mathbb R$-Cartier 
divisor and let $M$ be a Cartier divisor 
on $X$ such that $N=M-\omega$ is ample. 
Let $x_{1}, x_{2}\in X$ be two closed points. We assume 
that there are positive numbers $c(k)$ with the following properties.
\begin{itemize}
\item[$(1)$] If $Z\subset X$ is an 
irreducible {\em{(}}positive dimensional{\em{)}} subvariety 
that contains $x_{1}$ or $x_{2}$, then
$$
N^{\dim Z} \cdot Z>c(\dim Z)^{\dim Z}.
$$
\item[$(2)$] The numbers $c(k)$ satisfy the inequality:
$$
\sum \limits_{k=1}\limits^{\dim X} 2^{1/k} \frac{k}{c(k)}\leq 1.
$$
\end{itemize}
Then $\mathscr O_X(M)$ has a global section separating $x_{1}$ and $x_{2}$.
\end{thm}

\begin{rem}\label{rem1.4} 
In Theorem \ref{thm1.3}, 
let 
$W_{1}$, $W_{2}$ be the minimal qlc stratum of $[X,\omega]$ 
passing through $x_{1}$, $x_{2}$ respectively. 
Possibly after switching $x_{1}$ and $x_{2}$,
we can assume that $\dim W_{1} \leq \dim W_{2}$.
We put $W=W_{1} \cup W_{2}$ with the reduced structure. 
Then, by adjunction, 
$[W, \omega|_W]$ has a natural quasi-log structure 
with only qlc singularities induced by $[X, \omega]$ 
(see Theorem \ref{thm2.7}). 
Moreover, we have 
$H^{1}(X, \mathscr{I}_{W}\otimes \mathscr{O}_{X}(M))=0$, 
where $\mathscr{I}_{W}$ is the defining ideal 
sheaf of $W$ on $X$ (see Theorem \ref{thm2.7}). 
Therefore, 
the natural restriction map 
$$H^{0}(X, \mathscr{O}_{X}(M))\rightarrow 
H^{0}(W, \mathscr{O}_{W}(M)) $$ 
is surjective. Thus, as in Remark \ref{rem1.2}, 
we can replace $X$ with $W$ in Theorem \ref{thm1.3}. 
\end{rem}

By \cite[Theorem 1.1]{fuji6}, 
we know that any quasi-projective  
semi-log canonical pair has a natural quasi-log 
structure with only qlc singularities, 
which is  
compatible with the original semi-log canonical 
structure. Therefore, Theorem \ref{thm1.1} and 
Theorem \ref{thm1.3} also hold for semi-log 
canonical pairs. 
For the precise statements, 
see Corollary \ref{cor3.5} and Corollary \ref{cor4.6}. 
Our proof of Theorem \ref{thm1.1} and Theorem \ref{thm1.3} 
works very well in the category of quasi-log schemes. 
On the other hand, it does not seem to 
work well in the category of 
semi-log canonical pairs. 
This is one of the key points of 
formulating and proving the effective freeness and effective 
point separation for qlc pairs. 

The paper is organized as follows.
In Section \ref{sec2}, we recall some basic definitions and 
properties of quasi-log schemes. Then 
we formulate and prove the inversion of adjunction for 
qlc pairs, which will play a crucial role in this paper. 
Section \ref{sec3} is devoted to the proof of Theorem \ref{thm1.1}. 
In Section \ref{sec4}, we prove Theorem \ref{thm1.3}. 
The proof of Theorem \ref{thm1.3} is essentially the 
same as the proof of Theorem \ref{thm1.1}. 

\begin{ack} 
The author would like to thank his advisor 
Professor Osamu Fujino
all the time for his great 
support and quite a lot of wonderful comments and useful discussions.
\end{ack}

We will work over $\mathbb C$, the complex number field, 
throughout this paper. 
We note that a {\em{scheme}} means a separated scheme 
of finite type over $\mathbb C$.  
For the details of the theory of quasi-log schemes, 
see \cite[Chapter 6]{fuji4} and \cite{fuji5}. 
We also note that \cite{fuji2} is a gentle introduction to 
the theory of quasi-log schemes. 

\section{Inversion of adjunction}\label{sec2}

Let us recall the definition of quasi-log schemes, which was 
first introduced by Ambro in \cite{ambro}. 
For the details of the theory of quasi-log schemes, 
see \cite[Chapter 6]{fuji4}.

\begin{defn}[{\cite[Chapter 6]{fuji4} 
and \cite[Definition 4.1]{ambro}}]\label{def2.1}
A quasi-log scheme is a scheme $X$ endowed with an 
$\mathbb{R}$-Cartier 
divisor (or $\mathbb R$-line bundle) 
$\omega$ on $X$, a proper closed subscheme 
$X_{-\infty}\subset X$, and a finite 
collection $\left\{C_{i} \right\}$ of reduced and irreducible 
subschemes 
of $X$ such that 
there exists a proper morphism $f: (Y, B_{Y})\rightarrow X$ 
from a globally embedded simple normal crossing 
pair satisfying the following properties: 
\begin{itemize}
\item[(1)] $K_{Y}+B_{Y}\sim_{\mathbb{R}} f^{*}\omega$. 
\item[(2)] The natural map 
$\mathscr{O}_{X} 
\rightarrow f_{*}\mathscr{O}_{Y}(\lceil -(B_{Y}^{<1}) \rceil)$
induces an isomorphism 
$$\mathscr{I}_{X_{-\infty}} \cong f_{*}\mathscr{O}_{Y}(\lceil 
-(B_{Y}^{<1}) \rceil-\lfloor B_{Y}^{>1} \rfloor),$$ 
where $\mathscr I _{X_{-\infty}}$ is the defining 
ideal sheaf of $X_{-\infty}$. 
\item[(3)] The collection of subvarieties 
$\left\{C_{i} \right\}$ coincides 
with the images of $(Y, B_{Y})$-strata that 
are not included in $X_{-\infty}$.
\end{itemize} 
The morphism $f:(Y, B_Y)\to X$ is usually called a 
quasi-log resolution of $[X, \omega]$. 
We sometimes use $\Nqlc (X, \omega)$ to denote 
$X_{-\infty}$.
If $X_{-\infty}=\emptyset$,
then we usually say that 
$[X,\omega]$ is a quasi-log canonical pair 
(a qlc pair, for short) or $[X, \omega]$ is a quasi-log 
scheme with only qlc singularities. 
\end{defn}

We note that $X$ may be reducible and is 
not necessarily equidimensional in Definition \ref{def2.1} 
(see Example \ref{exmp2.5} below). 
We give some remarks on Definition \ref{def2.1}. 

\begin{rem}\label{rem2.2} 
Definition \ref{def2.1} may look slightly different from 
the definition in \cite{ambro}. 
In \cite{ambro}, 
Ambro only assumes that 
$(Y, B_{Y})$ is an {\em{embedded normal crossing pair}} 
(see \cite[Section 2]{ambro} and 
\cite[Chapter 5]{fuji4} for the definitions and examples
of {\em{embedded normal crossing pair}} and see  
\cite[Appendix]{fuji5} for the difference between 
{\em{embedded normal crossing pair}} and 
{\em{globally embedded simple normal crossing 
pair}}). By \cite{fuji4} and \cite{fuji5}, we see that 
Definition \ref{def2.1} is equivalent to 
the original definition in \cite{ambro}. 
\end{rem}

\begin{rem}\label{rem2.3} 
By \cite[Remark 4.2]{ambro}, $[X, \omega]$ is a qlc 
pair if and only if the 
coefficients of $B_{Y}$ are $\leq 1$, that is, $B_{Y}$ 
is a subboundary $\mathbb R$-divisor. In this case, we have 
$\mathscr{O}_{X} \cong f_{*}\mathscr{O}_{Y}(\lceil -(B_{Y}^{<1}) 
\rceil)$. 
In particular, we see that $f$ is surjective 
and $\mathscr{O}_{X}\cong f_{*}\mathscr{O}_{Y}$. 
Therefore, $f$ has connected fibers and $X$ is seminormal. 
In particular, $X$ is reduced.
\end{rem}

\begin{rem}\label{rem2.4} 
The subvariety $C_i$ in Definition \ref{def2.1} is called 
a {\em{qlc stratum}} of $[X, \omega]$. 
A {\em{qlc center}} of $[X, \omega]$ 
means a qlc stratum of $[X, \omega]$ 
which is not an irreducible 
component of $X$.
\end{rem}

We give some examples of 
qlc pairs to see why the notion of qlc pairs 
is very important. 

\begin{exmp}\label{exmp2.5} 
Every log canonical pair $(X, \Delta)$ 
defines a natural quasi-log structure on 
$[X,K_{X}+\Delta]$ to make $[X,K_{X}+\Delta]$ 
a qlc pair. This idea played a crucial role in \cite{kk} 
to prove that log canonical pairs have only Du Bois 
singularities. 
Let $\{C_i\}_{i\in I}$ be the set of log canonical 
centers of $(X, \Delta)$. 
We put $W=\bigcup _{i\in J} C_i$ for any $\emptyset 
\ne J\subset I$ with 
the reduced structure. 
Then, by adjunction (see Theorem \ref{thm2.7}), 
$[W, (K_X+\Delta)|_W]$ has a natural 
quasi-log structure with only qlc singularities induced by 
$[X, K_X+\Delta]$.
\end{exmp}

\begin{exmp}\label{exmp2.6} 
A quasi-projective semi-log canonical pair 
$(X, \Delta)$ 
has a natural quasi-log structure 
on $[X,K_{X}+\Delta]$ to make $[X,K_{X}+\Delta]$ a qlc pair. 
For the details, see \cite[Theorem 1.1]{fuji6}. 
\end{exmp}

The above examples show that 
we need the theory of quasi-log schemes to 
understand log canonial pairs and semi-log canonical pairs 
deeply. 

The following theorem is one of the key results of the 
theory of quasi-log schemes, which heavily 
depends on 
the theory of mixed Hodge structures on cohomology with 
compact support.

\begin{thm}[{\cite[Theorems 4.4 and 7.3]{ambro} 
and \cite[Chapter 6]{fuji4}}]\label{thm2.7}
Let $[X,\omega]$ be a 
quasi-log scheme and let $X'$ be the union of $X_{-\infty}$ with a 
$($possibly empty$)$ union 
of some qlc strata of $[X,\omega]$. 
Then we have the following properties.
\begin{itemize}
\item[$(1)$] {\em{(Adjunction).}} 
Assume that $X'\neq X_{-\infty}$. 
Then $X'$ is a quasi-log scheme with $\omega'=\omega \vert_{X'}$ 
and $X'_{-\infty}=X_{-\infty}$. Moreover, the qlc 
strata of $[X',\omega']$ are exactly the qlc strata 
of $[X,\omega]$ that are included in $X'$.
\item[$(2)$] {\em{(Vanishing theorem).}} Assume 
that $ \pi: X \rightarrow S$ is a proper 
morphism between schemes. 
Let $L$ be a Cartier divisor on $X$ such that 
$L-\omega$ is nef and log big over $S$ 
with respect to $[X,\omega]$. 
Then $R^{i}\pi_{*}(\mathscr{I}_{X'}\otimes \mathscr{O}_{X}(L))=0$ 
for every $i>0$, where $\mathscr{I}_{X'}$ is the 
defining ideal sheaf of $X'$ on $X$.
\end{itemize}
Note that an $\mathbb{R}$-Cartier 
divisor is called nef and log big over $S$ 
with respect to $[X, \omega]$ 
if it is nef and big over $S$ and big over $S$ 
on every qlc stratum of $[X, \omega]$ when it is 
restricted to that stratum. 
\end{thm}

We will use the following two lemmas in the 
proof of Theorem \ref{thm1.1} and Theorem \ref{thm1.3}. 

\begin{lem}\label{lem2.8} 
Let $[X,\omega]$ be a qlc pair and 
 $B$ be an effective $\mathbb{R}$-Cartier divisor 
 on $X$. Assume that $\Supp B$ contains 
 no qlc centers of $[X,\omega]$. 
 Then $[X,\omega+B]$ has a 
 natural quasi-log structure induced by $[X,\omega]$.
\end{lem}
\begin{proof}
Let $f: (Y, B_{Y})\rightarrow X$ be 
a quasi-log resolution, where $(Y, B_{Y})$ 
is a globally embedded simple normal crossing 
pair. By taking further blow-ups,  
we can assume 
that $(Y, B_{Y}+f^{*}B)$ is a globally embedded simple normal 
crossing pair.
We note that $B_{Y}^{=1}$ and 
$\Supp f^{*}B$ have no common components by the assumption. 
We have $K_{Y}+B_{Y}+f^{*}B\sim_{\mathbb{R}} f^{*}(\omega+B)$. 
We put 
\begin{equation*}
\begin{split}
\mathscr{I}_{\Nqlc(X, \omega+B)} &= f_{*}
\mathscr{O}_{Y}(\lceil -(B_{Y}+f^{*}B)^{<1} 
\rceil-\lfloor (B_{Y}+f^{*}B)^{>1} \rfloor)
\\
& \subset f_{*}\mathscr{O}_{Y}
(\lceil -(B_{Y}^{<1}) \rceil)= \mathscr O_X. 
\end{split}
\end{equation*}
It gives the so-called ideal sheaf of the non-qlc locus 
$\Nqlc(X, \omega+B)$. 
By construction, there is 
a collection of subschemes $\left\{\overline{C}_{i} \right\}$ 
coincides with the image of $(Y, B_{Y}+f^{*}B)$-strata. Note that 
$\left\{\overline{C}_{i} \right\}\supset \left\{C_{i} \right\}$ by 
construction, where $\{C_i\}$ is 
the set of qlc strata of $[X, \omega]$.
The above conditions 
give a natural quasi-log structure of $[X,\omega+B]$.
\end{proof}

\begin{lem}[{\cite[Lemma 
4.6]{fuji7}}]\label{lem2.9} 
Let $[X,\omega]$ be a qlc pair such that 
$X$ is irreducible. Let $B$ be an effective $\mathbb{R}$-Cartier 
divisor on $X$. 
Then $[X,\omega+B]$ has a 
natural quasi-log structure, which coincides 
with the original quasi-log structure of $[X,\omega]$ 
outside $\Supp B$.
\end{lem}

The following theorem was suggested by Fujino, 
which is one of the main ingredients of the proof of 
Theorem \ref{thm1.1} and Theorem \ref{thm1.3} just 
as \cite[Theorem 1]{ohsawa-takegoshi} playing a 
crucial role in \cite{as} in the original analytic 
case and \cite[Theorem]{kawakita} 
in \cite{fuji1} for log canonical pairs. 

\begin{thm}[Inversion of Adjunction, 
Osamu Fujino]\label{thm2.10}
Let $[X,\omega]$ be a qlc pair 
and $B$ be an effective $\mathbb{R}$-Cartier 
divisor on $X$ such that 
$\Supp B$ contains no qlc centers of $[X,\omega]$. 
Let $X'$ be a union of some qlc strata of 
$[X,\omega]$, Then $[X,\omega+B]$ is 
qlc in a neighborhood of $X'$ if and only 
if $[X', \omega|_{X'}+B|_{X'}]$ is qlc.
\end{thm}

Before we prove Theorem \ref{thm2.10}, we have an 
important remark. 

\begin{rem}\label{Rem2.11} 
In Theorem \ref{thm2.10}, 
$[X', \omega|_{X'}]$ is a qlc pair by 
adjunction (see Theorem \ref{thm2.7}). 
By assumption, we see that $B|_{X'}$ contains 
no qlc centers of $[X', \omega|_{X'}]$. 
Then, by Lemma \ref{lem2.8}, 
we have a natural quasi-log structure on $[X', \omega|_{X'} 
+B|_{X'}]$. 
\end{rem}
\begin{proof}[Proof of Theorem \ref{thm2.10}] 
We take a quasi-log resolution $f:(Z, \Delta_Z)\to X$, 
where $(Z, \Delta_Z)$ is a globally embedded 
simple normal crossing pair. 
By taking some suitable 
blow-ups, we may assume that the union of all strata 
of $(Z, \Delta_Z)$ mapped to $X'$, which 
is denoted by $Z'$, 
is a union of some irreducible components of $Z$. 
We put $(K_Z+\Delta_Z)|_{Z'}=K_{Z'}+\Delta_{Z'}$ and 
$Z''=Z-Z'$. 
Without loss of generality, 
we may assume that 
$(Z, \Delta_Z+f^*B)$ is also a globally 
embedded simple normal crossing pair. 
We put $\Theta_Z=\Delta_Z+f^*B$. By adjunction, 
$f:(Z', \Delta_{Z'})\to X'$ is a quasi-log resolution 
of $[X', \omega|_{X'}]$. 
We put $K_{Z'}+\Theta_{Z'}=(K_Z+\Theta_Z)|_{Z'}$. 
First, we may assume that $[X, \omega+B]$ is qlc in a neighborhood 
of $X'$. 
Then $\Theta_{Z'}$ is a subboundary $\mathbb R$-divisor. 
By construction, 
$f:(Z', \Theta_{Z'})\to X'$ is a quasi-log resolution 
of $[X', \omega|_{X'}+B|_{X'}]$. 
Therefore, $[X', \omega|_{X'}+B|_{X'}]$ is also qlc. 
Next, we assume that 
$[X, \omega+B]$ is not qlc in a neighborhood 
of $X'$. By replacing $B$ with 
$(1-\varepsilon)B$ for $0<\varepsilon \ll 1$, we may assume that 
$\Delta_Z^{=1}=(\Delta_Z+f^*B)^{=1}=\Theta_Z^{=1}$. 
Note that 
$$
\mathscr I _{\Nqlc(X, \omega+B)}=f_*\mathscr O_Z 
(\lceil- (\Theta_Z^{<1})\rceil-\lfloor \Theta_Z^{>1}\rfloor)
$$ 
by definition. 
We put 
$$
\widetilde X =X' \cup \Nqlc (X, \omega+B)
$$ 
and consider the quasi-log structure of $[\widetilde X, 
(\omega+B)|_{\widetilde X}]$ induced by $[X, \omega+B]$. 
Then, by adjunction, we obtain 
$$
\mathscr I _{\Nqlc(\widetilde X, (\omega+B)|_{\widetilde X})}
=f_*\mathscr O_{Z'} (\lceil -(\Theta_{Z'}^{<1})\rceil -
\lfloor \Theta_{Z'}^{>1}\rfloor). 
$$  
We note that $\Delta_Z^{=1}=\Theta_Z^{=1}$ (see Remark \ref{rem2.12} 
below). 
By assumption, $\mathscr I _{\Nqlc(\widetilde X, 
(\omega+B)|_{\widetilde X})}$ is nontrivial on $X'$ 
because 
$\Nqlc (X, \omega+B)=\Nqlc (\widetilde X, (\omega+B)|_{\widetilde X})$. 
By construction, 
we can see that $f:(Z', \Theta_{Z'})\to X'$ is also 
a quasi-log resolution of 
$[X', \omega|_{X'}+B|_{X'}]$. 
Therefore, 
$$
\mathscr I_{\Nqlc(X', \omega|_{X'}+B|_{X'})}=\mathscr I_{\Nqlc 
(\widetilde X, (\omega+B)|_{\widetilde X})}
$$ 
is nontrivial. 
Thus, we obtain that $[X', \omega|_{X'}+B|_{X'}]$ is not qlc. 
\end{proof}

\begin{rem}\label{rem2.12}
In the proof of Theorem \ref{thm2.10}, 
we may assume that 
$\Delta_Z^{=1}=\Theta_Z^{=1}$ by replacing 
$B$ with $(1-\varepsilon )B$ for $0<\varepsilon \ll 1$. 
By this condition $\Delta_Z^{=1}=\Theta_Z^{=1}$, 
the union of all strata of $(Z, \Theta_Z)$ mapped to 
$\widetilde X$ is $Z'$. 
Therefore, $f: (Z', \Theta_{Z'})\to \widetilde X$ is 
a quasi-log resolution of $[\widetilde X, (\omega+B)|_{\widetilde X}]$. 
This is a key point of the proof of Theorem \ref{thm2.10}. 
\end{rem}

\section{Proof of Theorem \ref{thm1.1}}\label{sec3}

In this section, we will prove Theorem \ref{thm1.1}. 
The main result of this section is as follows. 

\begin{prop}[{\cite[Theorem 6.4]{kollar}, 
\cite[Proposition 2.1]{fuji1}}]\label{prop3.1}
Let $[X,\omega]$ be a projective 
qlc pair such that $\omega$ is an 
$\mathbb R$-Cartier 
divisor. Assume 
that $X$ is irreducible. Let $N$ be 
an ample $\mathbb R$-divisor 
on $X$ and $x\in X$ be a closed point.  
Assume that there are positive numbers $c(k)$ for 
$1\leq k\leq \dim X$ with the following properties. 
\begin{itemize}
\item[$(1)$] If $x\in Z\subset X$ is an irreducible 
{\em{(}}positive dimensional{\em{)}} subvariety, then
$$
N^{\dim Z} \cdot Z>c(\dim Z)^{\dim Z}.
$$
\item[$(2)$] The numbers $c(k)$ satisfy the inequality:
$$
\sum \limits_{k=1}\limits^{\dim X} \frac{k}{c(k)}\leq 1.
$$
\end{itemize}
Then there is an effective $\mathbb R$-Cartier divisor 
$D\sim_{\mathbb R} cN$ with $0\leq c<1$ and 
an open neighborhood $x\in X^{0} \subset X$ such that
\begin{itemize}
\item[(i)]  $[X^{0},(\omega+D)\vert_{X^{0}}]$ is qlc, and
\item[(ii)] $x$ is a qlc center of $[X^{0},(\omega+D)\vert_{X^{0}}]$.
\end{itemize}
Note that $[X, \omega+D]$ has a natural quasi-log structure by Lemma 
\ref{lem2.9}. 
\end{prop}

To prove this proposition, we need some preparations.

\begin{lem}\label{lem3.2}
Let $[X,\omega]$ be an irreducible qlc pair and 
$x\in X$ be a general smooth point. 
Let $B_{x}$ be an effective $\mathbb{R}$-Cartier 
divisor such that $\mathrm{mult}_{x}B_{x} > \dim_{x}X$. 
Then $[X, \omega+B_{x}]$ 
is not qlc at $x$.
\end{lem}

\begin{proof} 
By Lemma \ref{lem2.9}, $[X, \omega+B_x]$ has a 
natural quasi-log structure. 
Let $f: (Y, B_{Y})\rightarrow X$ 
be a quasi-log resolution of $[X, \omega]$. 
Since $x$ is a general smooth 
point, we may assume that every 
stratum of $(Y, \Supp B_{Y})$ is 
smooth over a nonempty Zariski open 
neighborhood $U_{x}$ of $x$.
We can 
assume that $U_{x}$ is smooth. 
By taking a blow-up 
along an irreducible component $E$ of $f^{-1}(x)$, we 
can directly check that 
$[X, \omega+B_{x}]$ is not qlc at $x$ 
by $\mathrm{mult}_x B_x>\dim _xX$. 
\end{proof}

\begin{prop}\label{prop3.3}
Let $X$ be a projective 
irreducible variety with $\dim X=n$. 
Let $\omega$ be 
an $\mathbb R$-Cartier divisor on $X$. 
Assume that there exists a nonempty 
Zariski open set $U\subset X$ such that 
$[U, \omega \vert_{U}]$ is a qlc pair. 
Let $H$ be an ample $\mathbb R$-divisor 
on $X$ such that $H^{n}>n^{n}$. 
Let $x$ be a closed point of $U$ such that 
no qlc centers of $[U, \omega\vert_{U}]$ contain $x$. 
Then there is an effective $\mathbb R$-Cartier divisor 
$B_{x}$ on $X$ such that $B_{x}\sim_{\mathbb R} H$ 
and that $[U, (\omega+B_{x})\vert_{U}]$ is not qlc at $x$. 
\end{prop}

\begin{proof}
Let us consider $X \times \mathbb{A}^{1}
\rightarrow \mathbb{A}^{1}$ and 
take a general irreducible curve $C'$ on $X \times \mathbb{A}^{1}$ 
passing through $(x,0)\in X \times \mathbb{A}^{1}$. 
Since $C'$ is a general curve,  $C'\rightarrow \mathbb{A}^{1}$ is finite.
Let $\nu: C\rightarrow C'$ be the normalization. 
By taking the 
base change of $X \times \mathbb{A}^{1}\rightarrow 
\mathbb{A}^{1}$ by $C\rightarrow C'\rightarrow \mathbb{A}^{1}$, 
we obtain $p_{2}: X \times C\rightarrow C$. By construction, 
there exists a section $s: C\rightarrow X \times C$ of $p_{2}$ 
such that $s(C)$ is passing through 
$(x,0)\in X \times C$ for some $0\in C$. 
By \cite[Lemma 12.2]{fuji3} 
(in which it was assumed that a variety should be normal, 
but the normality is not used in the proof), 
we can find an effective $\mathbb R$-Cartier divisor $B$ 
on $X \times C$ such that $B\sim_{\mathbb R} p_{1}^{*}H$, 
where $p_{1}: X \times C\rightarrow X$ is the first projection
and that $\mathrm{mult}_{s(C)}B > n$.
By shrinking $C$, we can assume 
that $B$ contains no fibers of $p_{2}$.
By shrinking $U$,  
we can further assume that $[U, \omega\vert_{U}]$ 
contains no qlc centers.
We consider the natural quasi-log structure on
$[U\times C, p_{1}^{*}\omega\vert_{U}+ 
p_{2}^{*}0+B\vert_{U\times C}]$ by Lemma \ref{lem2.8}. 
Note that $p_2^*0\cong U$ is a qlc center of this quasi-log structure. 
We assume that $[U, (\omega+B_x)\vert_{U}]$ is qlc at $x$, 
where $B_x=B|_{p_2^*0}$. 
By applying the inversion of adjunction (see Theorem \ref{thm2.10}) to 
$[U\times C, p_{1}^{*}\omega\vert_{U}+ 
p_{2}^{*}0]$, $B\vert_{U\times C}$, and 
$p_2^*0\cong U$, we see that 
$[U\times C, p_1^*\omega|_U+p_2^*0+B|_{U\times C}]$ is qlc 
in a neighborhood of $(x, 0)$ since $[U, (\omega+B_x)|_{U}]$ is 
qlc at $x$. 
Then we obtain that 
$[U\times C, p_1^*\omega|_U+p_2^*0+B|_{U\times 
C}]$ is qlc at $(s(t), t)$ if $t$ is sufficiently 
close to $0\in C$. 
Thus, by adjunction, $[U, (\omega+B|_{p_2^*t})|_U]$ 
is qlc at $s(t)$ if $t$ is 
sufficiently close to $0\in C$ and is general in $C$. 
On the other hand, $[U, (\omega+B|_{p_2^*t})|_U]$ is 
not qlc at $s(t)$ for general $t\in C$ by Lemma \ref{lem3.2}. 
This is a contradiction. 
Therefore, $[U, (\omega+B_x)|_U]$ is not 
qlc at $x$. This means that $B_x$ is a desired effective 
$\mathbb R$-Cartier divisor. 
\end{proof}

The following proposition was 
established for kawamata log terminal 
pairs in \cite[Theorem 6.7.1]{kollar} and 
for log canonical pairs in \cite[Proposition 2.7]{fuji1}. 

\begin{prop}\label{prop3.4}
Let $X$ be a projective 
irreducible variety and $\omega$ be 
an $\mathbb R$-Cartier divisor on $X$. 
Assume that there exists a nonempty 
Zariski open set $U\subset X$ such that 
$[U, \omega \vert_{U}]$ is a qlc pair.
Let $x\in U$ be a closed point and $Z$ be the minimal 
qlc stratum of $[U, \omega \vert_{U}]$ passing through $x$ 
with $k=\dim Z>0$. 
Let $H$ be an ample $\mathbb R$-divisor 
on $X$ such that $H^{k} \cdot \overline{Z}>k^{k}$, 
where $\overline {Z}$ is the closure of $Z$ in $X$. Then 
there are an effective $\mathbb R$-Cartier 
divisor $B\sim_{\mathbb R} H$, 
a real number $0<c<1$, and an 
open neighborhood $x\in X^{0}\subset U$ such that: 
\begin{itemize}
\item[$(1)$] $[X^{0},(\omega+cB)\vert_{X^{0}}]$ is qlc, and
\item[$(2)$] there is a minimal qlc stratum $Z_{1}$ of 
$[X^{0},(\omega+cB)\vert_{X^{0}}]$ passing through $x$ with 
$\dim Z_{1}<\dim Z$.
\end{itemize}
\end{prop}

\begin{proof}
Since $[U, \omega \vert_{U}]$ is a qlc pair, 
$[Z, \omega|_Z]$ has  a qlc structure by adjunction (see 
Theorem \ref{thm2.7}). 
Note that $Z$ is normal at $x$ because $Z$ is minimal. 
Of course, no qlc centers of $[Z, \omega|_Z]$ contain $x$. 
By Proposition \ref{prop3.3}, there is 
an effective $\mathbb R$-Cartier 
divisor $F_{Z}\sim_{\mathbb R} mH\vert_{\overline {Z}}$  
such that $[Z,(\omega+\frac{1}{m}F_{Z})\vert_{Z}]$ is 
not qlc at $x$. 
Furthermore, as in \cite[Lemma 12.2]{fuji3}, 
we can assume that $H=H_{1}+a_{2}H_{2}+\cdots
+a_{t}H_{t}$ where $H_{1}$ is an ample 
$\mathbb Q$-divisor such that 
$H_{1}^{k} \cdot \overline{Z}>k^{k}$, $a_{i}$ is a positive 
real number 
and $H_{i}$ is an ample Cartier divisor for every $i\geq 2$, 
and that $F_{Z}
=F_{1}+a_{2}F_{2}+\cdots+a_{t}F_{t}$ 
with $F_{i}\sim_{\mathbb Q} mH_{i}\vert_{\overline {Z}}$ 
for every $i$. 
By replacing $m$ and $F_Z$ with $mk$ and 
$kF_Z$ for some large positive 
integer $k$, we can take $m$ 
as large as we want. 
Especially, for every $i$, we can 
find $m$ 
such that $mH_{i}$ is an ample Cartier divisor 
and that $F_i\sim mH_i|_{\overline Z}$ for every $i$. 
We may further assume that 
$$
H^{1}(X, \mathscr{I}_{\overline {Z}}\otimes \mathscr{O}_{X}(mH_{i}))=0
$$
for every $i$ 
by Serre's vanishing theorem, 
where $\mathscr{I}_{\overline {Z}}$ is the 
ideal sheaf of $\overline {Z}$ on 
$X$, and that $\mathscr I_{\overline Z}\otimes 
\mathscr O(mH_i)$ is globally generated for every $i$. 
By the following short exact sequence:
$$
0\rightarrow \mathscr{I}_{\overline {Z}}\otimes 
\mathscr{O}_{X}(mH_{i}) \rightarrow  
\mathscr{O}_{X}(mH_{i}) \rightarrow \mathscr{O}
_{\overline{Z}}(mH_{i}) \rightarrow 0, 
$$
we obtain that the natural restriction map 
$$H^{0}(X, \mathscr{O}_{X}(mH_{i}))
\rightarrow H^{0}(\overline {Z}, \mathscr{O}_{\overline{Z}}(mH_{i}))$$ 
is surjective. Therefore, we can take $D_{i}\in \left| 
mH_{i} \right|$ on $X$ such that 
$D_{i}\vert_{\overline{Z}}=F_{i}$ for every $i$. 
We put $F=D_{1}+a_{2}D_{2}+
\cdots+a_{t}D_{t}$. 
Then $F\vert_{\overline {Z}}=F_{Z}$ and 
$F\sim_{\mathbb R}mH$. 
Since $Z$ is a minimal qlc 
stratum of $[U, \omega|_U]$ passing through $x$, 
in a neighborhood 
$x\in X^{0}\subset U$, we may assume that $\Supp F$ 
contains no qlc centers of 
$[U, \omega \vert_{U}]$. By the inversion of adjunction (see 
Theorem \ref{thm2.10}), 
$[X^{0},(\omega+\frac{1}{m}F)\vert_{X^{0}}]$ 
is not qlc at $x$. 
Since we assumed that 
$\mathscr I_{\overline Z}\otimes \mathscr O_X(mH_{i})$ is 
globally generated for every $i$, by 
choosing $D_i$ general for every $i$, 
we obtain that 
$[X^{0},(\omega+\frac{1}{m}F)\vert_{X^{0}}]$ 
is qlc on $X^{0} \setminus Z$. 
By the above argument, we have 
constructed an $\mathbb R$-Cartier 
divisor $F\sim_{\mathbb R}mH$ on $X$ 
such that
\begin{itemize}
\item[(1)] $F\vert_{\overline {Z}}=F_{Z}$;
\item[(2)] $[X^{0},(\omega+\frac{1}{m}F)
\vert_{X^{0}}]$ is qlc on $X^{0} \setminus Z$;
\item[(3)] $[X^{0},(\omega+\frac{1}{m}F)
\vert_{X^{0}}]$ is qlc at the generic point of $Z$;
\item[(4)] $[X^{0},(\omega+\frac{1}{m}F)
\vert_{X^{0}}]$ is not qlc at $x \in Z$.
\end{itemize}
We put $B=\frac{1}{m}F$. Let $c$ be 
the maximal real number such that 
$[X^{0},(\omega+cB)\vert_{X^{0}}]$ is 
qlc at $x$. Then, after shrinking $X^{0}$ 
further, we have a new minimal qlc center 
$Z_{1}$ of $[X^{0},(\omega+cB)\vert_{X^{0}}]$ 
passing through $x$. 
Note that $Z$ (after restriction) is also a qlc 
center in this new qlc structure too. Therefore, 
we have that $x\in Z_{1}\subset Z$ and $\dim Z_{1}<\dim Z$.
\end{proof}

Now, we are ready to prove Proposition \ref{prop3.1}. 

\begin{proof}[Proof of Proposition \ref{prop3.1}] 
Let $Z_{1}$ be the minimal qlc stratum of $[X, \omega]$ 
passing 
through $x$. If $\dim Z_{1}=0$, then it is 
done. If $\dim Z_{1}>0$, then, by 
Proposition \ref{prop3.4}, we 
can find $x\in D_{1}\sim_{\mathbb R} 
\frac{k_{1}}{c(k_{1})}N$,  $0<c_{1}<1$ and an open 
neighborhood $X^{0}$ of $x$ such 
that $[X^{0},(\omega+c_{1}D_{1})\vert_{X^{0}}]$ 
is qlc and $k_{2}=\dim Z_{2}<k_{1}$ where $Z_{2}$ is 
the minimal qlc strutum of $[X^{0},(\omega+c_{1}D_{1})\vert_{X^{0}}]$ 
passing through $x$. 
By Lemma \ref{lem2.9}, we can consider 
the natural quasi-log structure 
$[X,\omega+c_{1}D_{1}]$, which coincides 
with the one 
of $[X^{0},(\omega+c_{1}D_{1})\vert_{X^{0}}]$ 
when restricting to $X^{0}$.
Repeat this argument 
by Proposition \ref{prop3.4}. 
Finally, we get a sequence $\dim X
\geq k_{1}>k_{2}>\cdots >k_{t}>0$, where $k_{i}\in \mathbb{Z}$ with the following properties: 
\begin{itemize}
\item[(1)] there is an effective 
$\mathbb{R}$-Cartier divisor 
$x\in D_{i}\sim_{\mathbb R} \frac{k_{i}}{c(k_{i})}N$ for every $i$;
\item[(2)] there is a real number 
$0<c_{i}<1$ for every $i$; 
\item[(3)] $[X,\omega+\sum \limits_{i=1}\limits^{t} c_{i}D_{i}]$ 
is qlc in a neighborhood $X^{0}$ of $x$; 
\item[(4)] $x$ is a qlc center of 
$[X,\omega+\sum \limits_{i=1}\limits^{t} c_{i}D_{i}]$.
\end{itemize}
We put $D=\sum \limits_{i=1}\limits^{t} c_{i}D_{i}$. 
Then $D$ has the desired properties. 
Note that 
$0\leq c=\sum \limits_{i=1}\limits^{t} 
c_{i}\frac{k_{i}}{c(k_{i})}<1$ and $D \sim_{\mathbb R} cN$.
\end{proof}

\begin{proof}[Proof of Theorem \ref{thm1.1}]
By Remark \ref{rem1.2}, 
we may assume that $X$ is irreducible. 
Let $D$ be an $\mathbb{R}$-Cartier divisor constructed in 
Proposition \ref{prop3.1}. 
By Lemma \ref{lem2.9}, we still have a 
natural quasi-log structure on $[X,\omega+D]$. 
Then, by the construction of $X^{0}$ in Proposition \ref{prop3.1}, 
$x$ is still a qlc center of this quasi-log structure on 
$[X,\omega+D]$ and is disjoint from $\Nqlc(X,\omega+D)$. 
We consider $\widetilde{X}=x \cup \Nqlc(X,\omega+D)$. 
By adjunction (see Theorem \ref{thm2.7}), 
$\widetilde{X}$ has a quasi-log structure induced by 
$[X, \omega+D]$. Now we have the following 
short exact sequence:
$$
0\rightarrow \mathscr{I}_{\widetilde{X}} 
\rightarrow \mathscr{O}_{X} 
\rightarrow \mathscr{O}_{\widetilde{X}} \rightarrow 0. 
$$
Since $M-(\omega+D)\sim_{\mathbb R} (1-c)N$ is ample, 
$H^{1}(X, \mathscr{I}_{\widetilde{X}} \otimes \mathscr{O}_{X}(M))=0$ 
by the vanishing theorem in Theorem \ref{thm2.7}. 
Therefore, the natural restriction map 
$$H^{0}(X, 
\mathscr{O}_{X}(M))\rightarrow H^{0} 
(\widetilde{X}, \mathscr{O}_{\widetilde{X}}(M)) $$ 
is surjective. 
Note that $x\cap \Nqlc(X,\omega+D)=\emptyset$, Thus we obtain that 
$$
H^{0}(X, \mathscr{O}_{X}(M))\rightarrow 
H^{0}(\widetilde{X}, 
\mathscr{O}_{\widetilde{X}}(M))\rightarrow 
H^{0}(x,\mathscr{O}_{\widetilde{X}}(M))=
\mathscr O_{\widetilde X}(M)\otimes \mathbb{C}(x)
$$ 
is surjective. This is what we wanted.
\end{proof}

As a direct consequence of Theorem \ref{thm1.1}, 
we have: 

\begin{cor}[Effective freeness for semi-log canonical 
pairs] \label{cor3.5}
Let $(X, \Delta)$ be a projective semi-log canonical pair.
Let $M$ be a Cartier divisor on $X$ such that 
$N=M-(K_X+\Delta)$ is ample. 
Let $x\in X$ be a closed point. 
We assume that there are positive 
numbers $c(k)$ with the following properties.
\begin{itemize}
\item[$(1)$] If $x\in Z\subset X$ is an irreducible {\em{(}}positive dimensional{\em{)}} subvariety, then
$$
N^{\dim Z} \cdot Z>c(\dim Z)^{\dim Z}.
$$
\item[$(2)$] The numbers $c(k)$ satisfy the inequality:
$$
\sum \limits_{k=1}\limits^{\dim X} \frac{k}{c(k)}\leq 1.
$$
\end{itemize}
Then $\mathscr O_X(M)$ has a global section not vanishing at $x$.
\end{cor}

\begin{proof} 
By \cite[Theorem 1.1]{fuji6}, 
we see that $[X, K_X+\Delta]$ has a natural 
quasi-log structure with only quasi-log canonical singularities, 
which is compatible with the original semi-log canonical structure 
of $(X, \Delta)$. 
For the details, see \cite{fuji6}. 
Then, by Theorem \ref{thm1.1}, 
we know that $\mathscr O_X(M)$ has a global section not 
vanishing at $x$. 
\end{proof}

\section{Proof of Theorem \ref{thm1.3}}\label{sec4}

In this section, we will prove Theorem \ref{thm1.3}. 
By Remark \ref{rem1.4}, 
we may assume that $X=W_{1} \cup W_{2}$, 
where $W_{1}$ (resp.~$W_{2}$) is the minimal 
qlc stratum of $[X,\omega]$ passing through 
$x_{1}$ (resp.~$x_{2}$) with $\dim W_{1} \leq \dim W_{2}$.
We put $V=W_{1} \cap W_{2}$. 
Then we obtain that either $V \subsetneq W_{2}$ or $V= W_{2}$. 
If $V \subsetneq W_{2}$, then $V$ 
is a union of qlc centers of $W_{2}$ and $x_{2} \notin V$ 
by \cite[Proposition 4.8]{ambro} (see also \cite{fuji4}). 
 
First, we slightly generalize Proposition \ref{prop3.4} as follows. 

\begin{prop}\label{prop4.1}
Let $X$ be a projective 
irreducible variety and $\omega$ be 
an $\mathbb R$-Cartier divisor on $X$. 
Assume that there exists a nonempty 
Zariski open set $U\subset X$ such that 
$[U, \omega \vert_{U}]$ is a qlc pair.
Let $x\in U$ be a closed point and $Z$ be the minimal 
qlc stratum of $[U, \omega \vert_{U}]$ passing through $x$ 
with $k=\dim Z>0$. Let $V$ be a qlc center of 
$[U, \omega \vert_{U}]$ disjoint from $x$.
Let $H$ be an ample $\mathbb R$-divisor 
on $X$ such that $H^{k} \cdot \overline{Z}>k^{k}$, 
where $\overline {Z}$ is the closure of $Z$ in $X$. Then 
there are an effective $\mathbb R$-Cartier 
divisor $B\sim_{\mathbb R} H$, 
a real number $0<c<1$, and an 
open neighborhood $x\in X^{0}\subset U$ such that: 
\begin{itemize}
\item[$(1)$] $[X^{0},(\omega+cB)\vert_{X^{0}}]$ is qlc, and
\item[$(2)$] there is a minimal qlc stratum $Z_{1}$ of 
$[X^{0},(\omega+cB)\vert_{X^{0}}]$ passing through $x$ with 
$\dim Z_{1}<\dim Z$.
\item[$(3)$] $V$ is still a qlc stratum of 
$[X^{0},(\omega+cB)\vert_{X^{0}}]$ disjoint from
$x$ after restrition.
\end{itemize}
\end{prop}

\begin{proof} 
In the proof of Proposition \ref{prop3.4}, 
we can choose $F$ so general that $V$ is not contained 
in $\Supp B$, where $B=\frac{1}{m}F$. 
Then, by Lemma \ref{lem2.9}, $V$ is still a qlc center 
of the new quasi-log structure 
of $[X^0, (\omega+cB) \vert_{X^0}]$.
\end{proof}

Following the same line as in 
Section \ref{sec3}, we have a generalization of 
Proposition \ref{prop3.1}. 
The proof of Proposition \ref{prop4.2} is obvious.

\begin{prop}\label{prop4.2}
Let $[X,\omega]$ be a projective 
qlc pair such that $\omega$ is an 
$\mathbb R$-Cartier 
divisor. Assume 
that $X$ is irreducible. Let $N$ be 
an ample $\mathbb R$-divisor 
on $X$ and $x\in X$ be a closed point.  
Let $V$ be a qlc center of 
$[X,\omega]$ disjoint from $x$.
Assume that there are positive numbers $c(k)$ for 
$1\leq k\leq \dim X$ with the following properties. 
\begin{itemize}
\item[$(1)$] If $x\in Z\subset X$ is an irreducible 
{\em{(}}positive dimensional{\em{)}} subvariety, then
$$
N^{\dim Z} \cdot Z>c(\dim Z)^{\dim Z}.
$$
\item[$(2)$] The numbers $c(k)$ satisfy the inequality:
$$
\sum \limits_{k=1}\limits^{\dim X} \frac{k}{c(k)}\leq 1.
$$
\end{itemize}
Then there is an effective $\mathbb R$-Cartier divisor 
$D\sim_{\mathbb R} cN$ with $0\leq c<1$ and 
an open neighborhood $x\in X^{0} \subset X$ such that
\begin{itemize}
\item[(i)]  $[X^{0},(\omega+D)\vert_{X^{0}}]$ is qlc, and
\item[(ii)] $x$ is a qlc center of $[X^{0},(\omega+D)\vert_{X^{0}}]$.
\item[(iii)] $V$ is still a qlc center of 
$[X^{0},(\omega+D)\vert_{X^{0}}]$ disjoint from 
$x$ after restriction.
\end{itemize}
Note that $[X, \omega+D]$ has a 
natural quasi-log structure by Lemma 
\ref{lem2.9} and $V$ is a qlc center of $[X, \omega+D]$.
\end{prop}

We give a proof of Theorem \ref{thm1.3} when $W_1\cap W_2
\subsetneq W_2$. 

\begin{proof}[Proof of Theorem \ref{thm1.3} when 
$V=W_1\cap W_2\subsetneq W_2$]
This proof is essentially the same 
as the proof of  Theorem \ref{thm1.1}. 
Here, we will prove that 
$\mathscr{I}_{W_{1}} \otimes  \mathscr O_X(M)$ 
has a global section not vanishing at $x_{2}$. 
Note that such a section obviously separates 
$x_{1}$ and $x_{2}$. 
Since we have a natural isomorphism 
$$H^{0}(X, 
\mathscr{I}_{W_{1}} \otimes \mathscr{O}_{X}(M)) \cong  
H^{0} (W_{2}, \mathscr{I}_{V}\otimes \mathscr{O}_{W_{2}}(M)) 
$$ 
by $X=W_1\cup W_2$ and $V=W_1\cap W_2$, 
we only need to prove that
$\mathscr{I}_{V}\otimes\mathscr{O}_{W_{2}}(M)$ has a 
global section not vanishing at $x_{2}$.
Let us consider the qlc pair 
$[W_{2},\omega \vert_{W_{2}}]$ by adjunction. 
We take an $\mathbb{R}$-Cartier divisor 
$D$ on $W_2$ constructed as in 
Proposition \ref{prop4.2}. Then 
$x_{2}$ is a qlc center of the induced new 
quasi-log structure on 
$[W_{2},\omega \vert_{W_{2}}+D]$ and 
is disjoint from $\Nqlc(W_{2},\omega \vert_{W_{2}}+D)$. 
Note that
$V$ is still a qlc center of this 
new quasi-log structure.
We put $\widetilde{W}=x_{2}\cup V \cup 
\Nqlc(W_{2},\omega \vert_{W_{2}}+D)$.
By adjunction (see Theorem \ref{thm2.7}), 
$\widetilde{W}$ has a quasi-log structure induced by 
$[W_{2},\omega \vert_{W_{2}}+D]$. Now we have the following 
short exact sequence:
$$
0\rightarrow  
\mathscr{I}_{\widetilde{W}} \rightarrow 
 \mathscr{O}_{W_{2}} 
\rightarrow  
\mathscr{O}_{\widetilde{W}} \rightarrow 0. 
$$
Since $M-(\omega+D)\sim_{\mathbb R} (1-c)N$ is ample,
$H^{1}(W_{2}, \mathscr{I}_{\widetilde{W}} 
\otimes \mathscr{O}_{W_{2}}(M))=0$ 
by the vanishing theorem in Theorem \ref{thm2.7}. 
Therefore, the natural restriction map 
$$H^{0}(W_{2}, 
 \mathscr{O}_{W_{2}}(M))\rightarrow 
H^{0} (\widetilde{W}, 
\mathscr{O}_{\widetilde{W}}(M)) 
$$ 
is surjective. 
We put $\widetilde{V}= \Nqlc(X,\omega+D) \cup V$. 
Then $x_{2} \cap \widetilde{V} =\emptyset$, 
Thus we obtain that 
$$ 
H^{0}(W_{2}, 
 \mathscr{O}_{W_{2}}(M))\rightarrow 
H^{0} (\widetilde{W}, 
\mathscr{O}_{\widetilde{W}}(M))=  
H^{0}(\widetilde{V},
\mathscr{O}_{\widetilde{V}}(M)) \oplus \mathbb{C}(x_{2})
$$ 
is surjective. By taking a pull-back 
of $0 \oplus 1$, we get what we wanted.
\end{proof}

\begin{rem}\label{rem4.3} 
In the case when $V\subsetneq W_2$, 
the assumptions in 
Theorem \ref{thm1.3} can be replaced as follows.
\begin{itemize}
\item[$(1)$] If $x_{2} \in Z\subset X$ is an 
irreducible (positive dimensional) subvariety, then
$$
N^{\dim Z} \cdot Z>c(\dim Z)^{\dim Z}.
$$
\item[$(2)$] The numbers $c(k)$ satisfy the inequality:
$$
\sum \limits_{k=1}\limits^{\dim W_{2}} \frac{k}{c(k)}\leq 1.
$$
\end{itemize}
This is obvious by the above proof of Theorem \ref{thm1.3} 
for the case when $V\subsetneq W_2$. 
\end{rem}

From now on, we treat the case when $V=W_1=W_2$. 

\begin{prop}\label{prop4.4}
Let $X$ be a projective 
irreducible variety with $\dim X=n$. 
Let $\omega$ be 
an $\mathbb R$-Cartier divisor on $X$. 
Assume that there exists a nonempty 
Zariski open set $U\subset X$ such that 
$[U, \omega \vert_{U}]$ is a qlc pair. 
Let $H$ be an ample $\mathbb R$-divisor 
on $X$ such that $H^{n}>2n^{n}$. 
Let $x_{1}, x_{2}$ be two closed points of $U$ such that 
no qlc centers of $[U, \omega\vert_{U}]$ contain $x_{1}, x_{2}$. 
Then there is an effective $\mathbb R$-Cartier divisor 
$B_{x_{1}, x_{2}}$ on $X$ such 
that $B_{x_{1}, x_{2}}\sim_{\mathbb R} H$ 
and that $[U, (\omega+B_{x_{1}, x_{2}})\vert_{U}]$ 
is not qlc at $x_{1}, x_{2}$. 
\end{prop}

\begin{proof}
As in the proof of Proposition \ref{prop3.3}, 
we can take a smooth pointed affine curve $0\in C$ such that 
the second projection $p_2: X\times C\to C$ has two sections 
$s_1$ and $s_2$, $s_1(C)$ (resp.~$s_2(C)$) is passing 
through $(x_1, 0)$ (resp.~$(x_2, 0)$) on $X\times C$, 
and $p_1(s_i(C))$ is a general curve on $X$ for $i=1, 2$, 
where $p_1$ is the first projection $X\times C\to X$. 
Then we can find an effective $\mathbb R$-Cartier divisor 
$B$ on $X\times C$ such that 
$B\sim _{\mathbb R} p_1^*H$ and 
$\mathrm{mult}_{s_i(C)}B >n$ for $i=1,2$. 
The same arguments as in the proof of Proposition \ref{prop3.3} 
produce a desired effective $\mathbb R$-Cartier divisor 
$B_{x_1, x_2}$ on $X$. 
We leave the details as an exercise for the reader. 
\end{proof}

Thus we can generalize 
Proposition \ref{prop3.4} to the following one
without any difficulties. 

\begin{prop}\label{prop4.5}
Let $X$ be a projective 
irreducible variety and $\omega$ be 
an $\mathbb R$-Cartier divisor on $X$. 
Assume that there exists a nonempty 
Zariski open set $U\subset X$ such that 
$[U, \omega \vert_{U}]$ is a qlc pair.
Let $x_{1}, x_{2} \in U$ be two closed points and $Z$ be the common minimal 
qlc stratum of $[U, \omega \vert_{U}]$ passing through $x_{1}, x_{2}$ 
with $k=\dim Z>0$. 
Let $H$ be an ample $\mathbb R$-divisor 
on $X$ such that $H^{k} \cdot \overline{Z}>2k^{k}$, 
where $\overline {Z}$ is the closure of $Z$ in $X$. Then 
there are an effective $\mathbb R$-Cartier 
divisor $B\sim_{\mathbb R} H$, 
a real number $0<c<1$, and an 
open neighborhood $x_{1}, x_{2}\in X^{0}\subset U$ such that: 
\begin{itemize}
\item[$(1)$] $[X^{0},(\omega+cB)\vert_{X^{0}}]$ 
is qlc at one point of $x_{1}, x_{2}$, say at $x_{1}$, 
\item[$(2)$] there is a minimal qlc stratum $Z_{1}$ of 
$[X^{0},(\omega+cB)\vert_{X^{0}}]$ passing through $x_{1}$ with 
$\dim Z_{1}<\dim Z$, and 
\item[$(3)$] $[U, (\omega+B)|_U]$ is not qlc at $x_2$ or there 
exists a 
qlc center of $[U, (\omega+B)|_U]$ passing through $x_2$. 
\end{itemize}
\end{prop}

\begin{proof}[Proof of Theorem \ref{thm1.3} 
when $V=W_1=W_2$]
We use Proposition \ref{prop4.5} to 
cut down $X=W_{1}= W_{2}$. Then we obtain 
$[X,\omega+cB]$, where $0<c<1$ is a real number, such that:
\begin{itemize}
\item[$(1)$] $B$ is an effective $\mathbb R$-Cartier 
divisor 
such that $B \sim_{\mathbb R} \frac{2^{1/k}k}{c(k)}N$, 
where $k=\dim X$.
\item[$(2)$] $[X,\omega+cB]$ is qlc at one point of $x_{1}, x_{2}$, 
say at $x_{1}$, 
\item[$(3)$] there is a minimal qlc stratum $\overline {Z_{1}}$ of 
$[X,\omega+cB]$ passing through $x_{1}$ with 
$\dim \overline {Z_{1}}<\dim Z$, and 
\item[(4)] $[X, \omega+cB]$ is not qlc at $x_2$ or there is 
a qlc center of $[X, \omega+cB]$ passing through $x_2$. 
\end{itemize}
If $[X,\omega+cB]$ is qlc at both $x_{1}$ and 
$x_{2}$ and if $x_{1}$ and $x_{2}$ are still stay on 
the same minimal qlc center of $[X,\omega+cB]$, then we apply 
Proposition \ref{prop4.5} again. 
By repeating this process finitely many times,
we will obtain the situation 
where there is a suitable 
effective $\mathbb R$-Cartier divisor $B'$, such that 
$[X,\omega+B']$ is not qlc at one of $x_{1}$ and $x_{2}$, or
$x_{1}$ and $x_{2}$ are on different minimal qlc centers of 
$[X,\omega+cB]$. Then we will go to the first case
proved by
Proposition \ref{prop4.1} (and Section \ref{sec3}). 
We leave the details as an exercise for the 
reader. 
Thus we get what we want. 
\end{proof}

As a corollary of Theorem \ref{thm1.3}, we have: 

\begin{cor}[Effective point separation for semi-log canonical 
pairs] \label{cor4.6}
Let $(X, \Delta)$ be a projective semi-log canonical pair.
Let $M$ be a Cartier divisor on $X$ such that 
$N=M-(K_X+\Delta)$ is ample. 
Let $x_{1}, x_{2}\in X$ be two closed points. We assume 
that there are positive numbers $c(k)$ with the following properties.
\begin{itemize}
\item[$(1)$] If $Z\subset X$ is an 
irreducible {\em{(}}positive dimensional{\em{)}} subvariety 
that contains $x_{1}$ or $x_{2}$, then
$$
N^{\dim Z} \cdot Z>c(\dim Z)^{\dim Z}.
$$
\item[$(2)$] The numbers $c(k)$ satisfy the inequality:
$$
\sum \limits_{k=1}\limits^{\dim X} 2^{1/k} \frac{k}{c(k)}\leq 1.
$$
\end{itemize}
Then $\mathscr O_X(M)$ has a global section separating $x_{1}$ and $x_{2}$.
\end{cor}
\begin{proof}
See the proof of Corollary \ref{cor3.5}.
\end{proof}

\end{document}